\documentclass[12pt]{amsart}
\usepackage{amsmath}
\usepackage{amsfonts}
\usepackage{amssymb}
\usepackage[all]{xy}
\usepackage[ top=1in, bottom=1in, left=1in, right=1in ]{geometry}
\usepackage[colorlinks, linkcolor=blue, urlcolor=blue]{hyperref}
\allowdisplaybreaks
\theoremstyle{plain}

\newtheorem{thm}{Theorem}[section]
\newtheorem{cor}[thm]{Corollary}

\newtheorem{claim}[thm]{Claim}

\newtheorem{lem}[thm]{Lemma}

\newtheorem{hypothesis}[thm]{Hypothesis}
\newtheorem{keylemma}[thm]{Key Lemma}
\newtheorem*{mainthm}{Main Theorem}
\newtheorem*{mainthmminors}{Main Theorem on Minors}
\newtheorem{main}[thm]{Main Theorem}
\newtheorem{mainminors}[thm]{Main Theorem on Minors}
\theoremstyle{definition}
\newtheorem{rem}[thm]{Remark}
\newtheorem{defn}[thm]{Definition}
\numberwithin{equation}{thm}

\newcommand{\field}[1]{\mathbb{#1}}
\newcommand{\inj}{\hookrightarrow}
\newcommand{\isom}{\cong}
\newcommand{\surj}{\twoheadrightarrow}
\newcommand{\tensor}{\otimes}

\newcommand{\p}{\mathfrak{p}}

\newcommand{\m}{\mathfrak{m}}

\newcommand{\Ann}{\operatorname{Ann}}

\newcommand{\height}{\operatorname{ht}}
\newcommand{\Hom}{\operatorname{Hom}}
\newcommand{\id}{\operatorname{id}}
\newcommand{\image}{\operatorname{Im}}
\newcommand{\Soc}{\operatorname{Soc}}
\newcommand{\stab}{\operatorname {stab}}

\newcommand{\Ass}[2]{\operatorname{Ass}_{#1} \left( #2 \right)}
\newcommand{\I}[2]{ I_{#1}(#2) }
\newcommand{\invpt}[1]{ \left( #1 \right)^G }

\newcommand{\LC}[3]{ H^{#1}_{#2}\left(#3\right) }

\newcommand{\dual}[1]{{#1}^\#}
\newcommand{\dbldual}[1]{{#1}^{\# \#}}
\newcommand{\findual}[1]{{#1}^*}
\newcommand{\dblfindual}[1]{{#1}^{**}}

\newcommand{\rat}{\text{rational}}
\newcommand{\Rat}{\text{Rational}}
\newcommand{\Gmod}{G\text{-module}}

\newcommand{\RGmod}{R[G]\text{-module}}

\newcommand{\ratGmod}{\rat\text{ }\Gmod}
\newcommand{\RatGmod}{\Rat\text{ }\Gmod}
\newcommand{\ratRGmod}{\rat\text{ }\RGmod}
\newcommand{\RatRGmod}{\Rat\text{ }\RGmod}

\newcommand{\edit}[1]{#1}

\begin{document}

\title[Local cohomology with support in minors]{Local cohomology with support in ideals of \\maximal minors}
\author{ Emily E. Witt }
\thanks{NSF support through grant DMS-0502170 is gratefully acknowledged.}

\begin{abstract}
Suppose that $k$ is a field of characteristic zero, $X$ is an $r \times s$ matrix of indeterminates, where $r \leq s$, and $R = k[X]$ is the polynomial ring over $k$ in the entries of $X$.  We study the local cohomology modules $H^i_I(R)$, where $I$ is the ideal of $R$ generated by the maximal minors of $X$.  We identify the indices $i$ for which these modules vanish, compute $H^i_I(R)$ at the highest nonvanishing index, $i = r(s-r) +1$, and characterize all nonzero ones as submodules of certain indecomposable injective modules.  These results are consequences of more general theorems regarding linearly reductive groups acting on local cohomology modules of polynomial rings.  
\end{abstract}

\maketitle

%%%%%%%%%%%%%%%%%%%%%%%%%%%%%%

\section{Introduction}

%%%%%%%%%%%%%%%

\subsection{History}
A major goal of this article is to understand local cohomology modules of polynomial rings with support in ideals generated by determinants.  More precisely, if $X = \left[ x_{ i j } \right]$ is an $r \times s$ matrix of indeterminates, where $r \leq s$, consider the polynomial ring $R$ over a field $k$ in the entries of $X$, i.e., $R = k[x_{i j} \ | \ 1 \leq i \leq r, 1 \leq j \leq s]$. We are concerned with understanding local cohomology modules of $R$ with support in the ideal $I$ generated by the maximal minors of $X$. %, $\LC{i}{I}{R}$.    

The behavior of these local cohomology modules depends strongly on the characteristic of the ring.  In prime characteristic, by results of Hochster and Eagon and of Peskine and Szpiro, there is only one nonzero such local cohomology module of the form $\LC{i}{I}{R}$, which has index $i = s-r+1$, the depth of $I$ \cite[Theorem 1]{Hoch-Eag}, \cite[Th\'eor\`eme III.4.1]{P-S}.  

In characteristic zero, this article's case of focus, the minimum index for which $\LC{i}{I}{R}\neq0$ is the same as in prime characteristic \cite[Theorem 1]{Hoch-Eag}.  However, an argument of Hochster, Huneke, and Lyubeznik shows that the maximum nonvanishing index is $r(s-r) + 1$, almost $r$ times larger \cite[Remark 3.13]{H-L}. 

The only previously known explicit description of such a local cohomology module in characteristic zero is due to Walther, who showed that when $X$ is a $2 \times 3$ matrix, the local cohomology module at the largest nonvanishing index, $\LC{3}{I}{R}$, is isomorphic to the injective hull of $k$ over $R$,  $E_R(k)$ \cite[Example 6.1]{Walther}.  His example motivates the question whether this phenomenon occurs in general for an $r \times s$ matrix; i.e., if $d = r(s-r)+1$ is the ``maximum nonvanishing" index, then is $\LC{d}{I}{R}$ always isomorphic to  $E_R(k)$?  

In computing this example, Walther employed a powerful theorem of Lyubeznik proved using the $D$-module structure of local cohomology modules. This result indicates that since $\LC{3}{I}{R}$ is supported only at the homogeneous maximal ideal of $R$, it is isomorphic to a finite direct sum of copies of  $E_R(k)$ \cite[Theorem 3.4]{Lyu}.  In general, it is easily checked that $\LC{d}{I}{R}$ is supported only at the homogeneous maximal ideal, so we know that $\LC{d}{I}{R}$ must again have this form.

%%%%%%%%%%%%%%%

\subsection{Main Results}

This article proves that $\LC{d}{I}{R}$ is isomorphic to exactly \emph{one} copy of the injective hull of $k$ over $R$.  Our method relies on invariant theory, as well as the work of Lyubeznik cited earlier.  The article also provides information about the local cohomology modules $\LC{i}{I}{R}$ at indices $i < d$.  Our main result regarding the local cohomology modules $\LC{i}{I}{R}$ in the characteristic zero case is the following:

\begin{mainthmminors}[\ref{LC}] \label{mtm}
\label{minors}
Let $k$ be a field of characteristic zero and let $X$ be an $r \times s$ matrix of indeterminates, where $r < s$.  Let $R = k[X]$ be the polynomial ring over $k$ in the entries of $X$, and let $I$ be its ideal generated by the maximal minors of $X$.  Given an $R$-module $M$, let $E_R(M)$ denote the injective hull of $M$ over $R$.

\begin{enumerate}
\item[\textup{(1)}] For $d= r(s-r)+1$,  $\LC{d}{I}{R} \cong E_R(k).$
\item[\textup{(2)}] $\LC{i}{I}{R} \neq 0$ if and only if $i=(r-t)(s-r) + 1$ for some $0 \leq t < r.$
\item[\textup{(3)}] Furthermore, if $i= (r-t)(s-r)+1$, then \[\LC{i}{I}{R} \hookrightarrow E_R(R/ I_{t+1}) \cong \LC{i}{I}{R}_{I_{t+1}},\] where $I_{t+1}$ is the ideal of $R$ generated by the $(t+1) \times (t+1)$ minors of $X$ \textup{(}which is prime by \emph{\cite[Theorem 1]{Hoch-Eag}}\textup{)}.  In particular, $ \Ass{R}{\LC{i}{I}{R}} = \{ \I{t+1}{X} \}.$
\end{enumerate}
\end{mainthmminors}

\noindent Note that there is precisely one nonvanishing local cohomology module of the form $\LC{i}{I}{R}$ for every possible size minor of $X$, and that each nonvanishing $\LC{i}{I}{R}$ injects into a specific indecomposable injective module.  Moreover, this result is proven independently of the result of Hochster, Huneke, and Lyubeznik cited earlier.

The proof of Main Theorem on Minors \ref{LC} takes advantage of the natural action of the group $G = SL_r(k)$ on the ring $R$.  The fact that this group also acts on each of the local cohomology modules is a powerful tool.  A classical result from invariant theory is that $R^G$, the subring of invariant elements of $R$, is the $k$-subalgebra of $R$ generated by the maximal minors of $X$ \cite[Theorem 2.6.A]{Weyl}.  This means that the ideal $I$ of $R$ generated by the maximal minors of $X$ is the expansion of the homogeneous maximal ideal of $R^G$ to $R$.  

This technique, in fact, can be extended more generally to a polynomial ring with a ``nice" action of any linearly reductive group.  Indeed, we prove the following more general theorem:

\begin{mainthm} [\ref{main}] Let $R$ be a polynomial ring over a field $k$ of characteristic zero with homogeneous maximal ideal $\m$.  Let $G$ be a linearly reductive linear algebraic group over $k$ acting by degree-preserving $k$-automorphisms on $R$, such that $R$ is a $\ratGmod$ \textup{(}see \emph{Definition \ref{regularGmod}}\textup{)}. Assume that $A = R^G$ has homogeneous maximal ideal $\m_A$, let $d = \dim A$, let $I = \m_A R$, and let $E_R(k)$ denote the injective hull of $k$ over $R$.  Then $\LC{d}{I}{R} \neq 0$ and $I$ is generated up to radicals by $d$ elements and not fewer, so that $\LC{i}{I}{R} = 0 \text{ for } i>d.$  Moreover, the following hold:

\begin{enumerate}
\item[\textup{(1)}] If $i < d$, then $\m$ is not an associated prime of $\LC{i}{I}{R}$, i.e., $\LC{0}{\m}{\LC{i}{I}{R}}=0$.
\end{enumerate}  

\noindent If, in addition, $\LC{d}{I}{R}$ is supported only at $\m$ \textup{(}e.g., this holds if, after localization at any of the indeterminates of $R$, $I$ requires fewer than $d$ generators up to radical\textup{)}, then
\begin{enumerate}
\item[\textup{(2)}]  If $H := \LC{d}{I}{R}$, then $V := \Soc H$ is a simple $G$-module, and 
\item[\textup{(3)}]  \edit{There exists a $G$-submodule $W$ of $H$ such that $H \cong \dual{ \left( R \tensor_k W \right) }$ as rational $R[G]$-modules, which, as $\ratGmod$s, is isomorphic to $E_R(k) \tensor_k V$ \textup{(}where the action of $G$ on $E_R(k) \cong \dual{R}$ is induced by its action on $R$\textup{)}.  \textup{(}See \textup{Definitions \ref{regularGmod}, \ref{reg}, \ref{gradeddual}.)}}
\end{enumerate}  
\end{mainthm}

%%%%%%%%%%%%%%%

\subsection{Outline}  Section \ref{G-mod} provides background material on $G$-modules and $R[G]$-modules, and Section \ref{moreprelim} presents additional preliminary definitions and lemmas regarding graded duals and $G$-modules.  In Section \ref{mainthm}, Key Lemma \ref{hardlemma} is a crucial structural lemma utilized to prove Main Theorem \ref{main}.  Proving the Main Theorem on Minors \ref{LC} is then the focus of Section \ref{mainthmminors}.

%%%%%%%%%%%%%%%%%%%%%%%%%%%%%%

\section{$G$-modules and $R[G]$-modules} \label{G-mod}

Here, we review the relevant theory of $G$- and $R[G]$-modules; our main reference is \cite{Borel}.  

\begin{defn}[Linear algebraic group]
A \emph{linear algebraic group} over a field $k$ is a Zariski-closed subgroup of $GL_n(k)$, for some positive integer $n$.
\end{defn}

\begin{defn}[$G$-module, $G$-module action, $G$-submodule, simple $G$-module, $G$-module homomorphism, $G$-equivariant map] \label{Gmod}Given a linear algebraic group $G$ over a field $k$, a \emph{$G$-module} is a $k$-vector space \edit{$V$} with a $k$-linear representation of $G$, a group homomorphism $G \to GL(V)$. The corresponding group action $G \times V \to V$ on a $G$-module is called its \emph{$G$-module action}.  A \emph{$G$-submodule} $W$ of $V$ is a $k$-vector subspace of $V$ that is stable under its $G$-module action.  A \emph{simple $G$-module} is a nonzero $G$-module that contains no proper nonzero $G$-submodules.  Given $G$-modules $V$ and $W$, a \emph{$G$-module homomorphism} $\phi: V \to W$ is a vector space homomorphism that is also \emph{$G$-equivariant}, which means that for all $g \in G$ and $v \in V$, $g \cdot \phi(v) = \phi(g \cdot v).$  The $k$-vector space of all such maps is denoted $\Hom_G(V,W)$.
\end{defn}

\begin{defn}[$\RatGmod$] \label{regularGmod}
Given a linear algebraic group $G$ over a field $k$, a finite-dimensional $G$-module $V$ is called a \emph{$\ratGmod$} if the action $G \times V \to V$ is a regular map of affine varieties over $k$.  An arbitrary (possibly infinite-dimensional) $G$-module is a \emph{$\ratGmod$} if it is a directed union of $G$-stable finite-dimensional $k$-vector subspaces that are themselves $\ratGmod$s.
\end{defn}

A $G$-stable subspace of a $\ratGmod$, a quotient of a $\ratGmod$, or a direct sum of $\ratGmod$s, is again a $\ratGmod$.  If $V$ and $W$ are $\rat$ $G$-modules, then $V \tensor_k W$ is a rational $G$-module with action defined by $g \cdot (v \tensor w) = g \cdot v \tensor g \cdot w$ on simple tensors.  If $V$ is also a finite-dimensional vector space, then $\Hom_k(V, W)$ is a $\ratGmod$ by $g \cdot f = g f g^{-1}$.  Moreover, by definition, the directed union of $\ratGmod$s is again a $\ratGmod$.

\begin{rem} \label{coordring}
Every linear algebraic group $G$ acts $\rat$ly on the coordinate ring $k[G]$, and every finite-dimensional $\rat$ $G$-module occurs as a $G$-submodule of $k[G]^{\oplus h}$ for some $h$.  Every finite-dimensional simple rational $G$-module occurs as a $G$-submodule of $k[G]$ \cite[discussion following Definition 2.23]{Fogarty}.
\end{rem}

\begin{defn}[Linearly reductive group] A linear algebraic group $G$ is called \emph{linearly reductive} if every finite-dimensional $\ratGmod$ splits into a direct sum of simple $G$-modules.  
\end{defn}

In particular, if $G$ is linearly reductive, every \edit{surjective} map of $\ratGmod$s splits.  Some examples of linearly reductive groups in characteristic zero are the general linear group (and, in particular, the multiplicative group of the field), the special linear group, the orthogonal group, the symplectic group, finite groups, and products of any of these.  In characteristic $p>0$, there are fewer linearly reductive groups; some examples are the multiplicative group of the field,  finite groups whose orders are not multiples of $p$, and products of these.

\begin{defn}[$W$-isotypical component] \label{isotyp} If $G$ is a linearly reductive group, $V$ is a $\rat$ $\Gmod$, and $W$ is a simple $\ratGmod$, the \emph{$W$-isotypical component} of $V$ is the direct sum of all $G$-submodules of $V$ isomorphic to $W$, i.e., it is of the form $\bigoplus \limits_i W_i \subseteq V$, where each $W_i \cong W$ as $G$-modules.  As a $G$-module, $V$ is the direct sum of its isotypical components.
\end{defn}

\begin{defn}[Invariant part] If $V$ is a $G$-module, then the \emph{invariant part of $V$}, denoted $V^G$, is the $G$-submodule of elements in $V$ fixed by the action of $G$.  
\end{defn}

\begin{rem} \label{invexact}
\edit{When $G$ is linearly reductive, $V^G$ is the isotypical component of $k$ with the trivial action, so the functor on $\ratGmod$s sending $V$ to $V^G$ is exact.  The sum of all other isotypical components (the sum of all the simple $G$-submodules of $V$ on which $G$ does not act trivially) defines a unique $G$-module complement of $V^G$.}
\end{rem}

\begin{defn}[$\RGmod$] \label{reg}
Let $G$ be a linear algebraic group over a field $k$ and let $R$ be  $k$-algebra that is a $G$-module.  An $R$-module $M$ that is also a $G$-module is an \emph{$R[G]$-module} if for every $g \in G$, $r \in R$, and $u \in M,$ $g(ru) = (gr)(gu).$
\end{defn}

\begin{defn}[$\RatRGmod$]
Given a field $k$ and a $k$-algebra $R$ with an action of a linear algebraic group $G$, a \emph{$\ratRGmod$} is an $R[G]$-module that is also a $\ratGmod$.  
\end{defn}

\begin{rem}
By Remark \ref{coordring}, every simple $\rat$ $G$-module occurs in the action of $G = SL_r(k)$ on $k[G] = k[x_{i j}]_{r \times r}/(\det ([x_{i j}]_{r \times r}) - 1),$ and hence in the action on $k[x_{i j}]_{r \times r}$, which maps onto $k[G]$.  Thus, all occur in the action on $k[X]$, where $X = [x_{i j}]_{r \times s}$ and $r \leq s$, which contains $k[x_{i j}]_{r \times r}.$
\end{rem}

The following isomorphism will be used to prove Lemma \ref{tensorinvariants} and Lemma \ref{isom}.

\begin{rem} \label{vsisom}
Given a linear algebraic group $G$ over a field $k$ and (rational) $G$-modules $U$ and $V$, $\dim_k V < \infty$, we have an isomorphism of (rational) $G$-modules  \begin{align} \label{isomrmk} U \tensor_k \findual{V} &\cong \Hom_k(V, U), \end{align} under which $u \tensor f \mapsto \phi$, where $\phi(v) = f(v) u$.  (Given $g \in G$, for $u \tensor f \in U \tensor_k \findual{V}$, $g \cdot (u \tensor f) = gu \tensor gf,$ and for $\phi \in \Hom_k(V, U)$ and $v \in V$, $(g  \cdot \phi)(v) = ( g \phi g^{-1} )(v).$)
\end{rem}

\begin{lem} \label{tensorinvariants}
If $G$ is a linearly reductive group over a field $k$, and $U$ and $W$ are $\rat$ $\Gmod$s, then $(U \tensor_k W)^G \neq 0$ if and only if, for some simple $G$-submodule $V$ of $U$, $\findual{V} \inj W$.
\end{lem}

\begin{proof}
For simple $G$-modules $U$ and $W$, $\invpt{U \tensor_k W} \neq 0$ if and only if $W \isom \findual{U}$ as $G$-modules:  
By Remark \ref{vsisom}, $U \tensor_k W \cong \Hom_k(\findual{U}, W)$ as $G$-modules.  Say $\invpt{U \tensor_k W} \neq 0$, so that there exists $0 \neq \phi \in \invpt{\Hom_k(\findual{U}, W)}$.  Since $U$ and $W$ are simple (so that $\findual{U}$ also is), $\image \phi = W$ and $\ker \phi = 0$, so $\phi$ is an isomorphism, and $\findual{U} \cong W$.  On the other hand, if $W \cong \findual{U}$, $U \tensor_k \findual{U} \cong \Hom_k(U, U)$ as $G$-modules.  Since for any $g \in G$ and $u \in U$, $(g \cdot \id_U) (u)= g \id_U g^{-1} (u) = \id_U (u)$, $\id_U \in \Hom_k(U, U)$ and corresponds to a nonzero element of $\invpt{U \tensor_k \findual{U}}$ under the isomorphism.

%By the adjointness of tensor and Hom, $\findual{\left(U \tensor_k V \right)} \cong \Hom_k\left(V, \findual{U}\right)$.  As $G$-modules,  $\findual{\left( \left(U \tensor_k V\right)^G\right)} \cong \left(\findual{\left(U \tensor_k V\right)}\right)^G \cong \invpt{ \Hom_k\left(V, \findual{U}\right) } \cong \Hom_G\left(V, \findual{U}\right)$.  If $\invpt{U \tensor_k V} \neq 0,$ the isomorphism gives a nonzero $G$-module map $V \to \findual{U},$ so since $U$ and $V$ (and so also $\findual{U}$) are simple $G$-modules, the map must be an isomorphism.  Now say $V \isom \findual{U}$, so that under \eqref{isomrmk}, $ U \tensor_k V \cong \Hom_k (U, U)$, where $G$ acts by conjugation.  The element of $U \tensor_k \findual{U}$ corresponding to the identity map in $\Hom_k (V, V)$ must be invariant since $g \id_U g^{-1} = \id_U$.

The general case now follows easily as $U$ and $V$ are direct sums of simple $G$-modules.
\end{proof}

\begin{cor} \label{isotypinv}
If $G$ is a linearly reductive group, $V$ is a simple $\ratGmod$, and $U$ is a $\ratGmod$ with $V$-isotypical component $\widehat{U}$, then $\left(  U \tensor_k \findual{V} \right)^G = \left(\widehat{U} \tensor_k \findual{V}\right)^G.$
\end{cor}

\begin{proof}
Suppose that $U = \bigoplus \limits_{i \in I} U_i$ as $G$-modules, where each $U_i$ is a simple $G$-module.  Let $J \subseteq I$ be the set of indices $j$ such that $U_j \cong V$ as $G$-modules, so that $\widehat{U} = \bigoplus \limits_{i \in J} U_i.$  For $i \in I - J$, by Lemma \ref{tensorinvariants}, $(  U_i \tensor_k \findual{V} )^G = 0.$ This means that $\invpt{U \tensor_k \findual{V}} = \invpt{ \left( \bigoplus \limits_{i \in J} U_i \right) \tensor_k \findual{V}  }$
\end{proof}

%%%%%%%%%%%%%%%%%%%%%%%%%%%%%%

\section{More Preliminaries} \label{moreprelim}

The $``\#"$ notation used in the following definition is not standard, but is very useful in our context.

\begin{defn}[Graded dual] \label{gradeddual}
If $k$ is a field and $V$ is a $\field{Z}$-graded $k$-vector space such that $\dim_k [V]_i < \infty$ for every $i \in \field{Z}$, then the \emph{graded dual of $V$}, \[\dual{V} = \bigoplus_{i\in \field{Z}} \Hom_k ([V]_i,k),\] is a $\field{Z}$-graded $k$-vector space satisfying $\left[ \dual{V} \right]_{j} = \Hom_k ([V]_{-j},k).$  
\end{defn}

\noindent Note that $\dual{(-)}$ is an exact contravariant functor.  Note also that if $V$ is a finite-dimensional $k$-vector space, then $ \dual{V} = \findual{V}:= \Hom_k (V, k).$  

\begin{rem}
Suppose that $k$ is a field and $R$ is an  $\field{N}$-graded ring finitely generated over $R_0 = k$ and homogeneous maximal ideal $\m$.   Suppose also that $M$ is a $\field{Z}$-graded Artinian $R$-module.  Then $M_{\geq i} := \bigoplus \limits_{n \geq i} M_n$ is a submodule of $M$, and since $\m M_{\geq i} \subseteq M_{\geq i + 1}$,  each $M_i \cong M_{\geq i}/ M_{\geq i+1}$ is a Noetherian $R$-module killed by $\m$, so is a finite-dimensional $k$-vector space.  Thus,  $M$ satisfies the hypotheses necessary to define its graded dual.  This is also true for $\field{Z}$-graded Noetherian $R$-modules, including $R$ itself.
\end{rem}

\begin{rem}[Graded duals, ($\rat$) $G$-modules, and ($\rat$) $\RGmod$s] \label{gmod}
\edit{Suppose that $G$ is a linear algebraic group over a field $k$.  Assume that $V$ is a $\field{Z}$-graded $G$-module such that $\dim_k [V]_i < \infty$ for every $i$, and that the action of $G$ preserves the grading of $V$.  Then $\dual{V}$ is also a $\field{Z}$-graded $G$-module:  
For any $g \in G,f \in \dual{V}$, and $v \in V$, $(gf)(v) = f(g^{-1}v),$ which is natural shorthand for $\sum \limits_i f_i(g^{-1}v_i)$, assuming $f = \sum \limits_i f_i$, where $\deg f_i = -i$, and  $v = \sum_i v_i$, where $\deg v_i = i$.
If $V$ is a $\ratGmod$, this action of $G$ makes $\dual{V}$ a $\ratGmod$ as well. }

\edit{Take $R$ an $\field{N}$-graded ring such that $R_0 = k$, and such that $R$ is a $G$-module so that the action of $G$ respects the grading of $R$.  
If $V$ is additionally a (rational) $R[G]$-module, then $\dual{V}$ is also a (rational) $R[G]$-module.}
%Suppose that $G$ is a linear algebraic group over a field $k$.  Assume that $R$ is an $\field{N}$-graded ring such that $R_0 = k$, and that $R$ is also a $G$-module, where $G$ acts on $R$ by $k$-automorphisms as to preserve its grading.  Assume that $M$ is a $\field{Z}$-graded $R[G]$-module such that $\dim_k [M]_i < \infty$ for every $i$, and that the action of $G$ preserves the grading of $M$. Then $\dual{M}$ is also a $\field{Z}$-graded $R[G]$-module.  For any $g \in G,f \in \dual{M}$, and $u \in M$, $(gf)(u) = f(g^{-1}u),$ which is natural shorthand for $\sum \limits_i f_i(g^{-1}u_i)$, assuming $f = \sum \limits_i f_i$, where $\deg f_i = -i$, and  $u = \sum_i u_i$, where $\deg u_i = i$.  If $M$ is a $\ratRGmod$, this action of $G$ makes $\dual{M}$ a $\ratRGmod$ as well. 
\end{rem}

\begin{rem} \label{dualinvcommute}
\edit{If $V$ is a $\field{Z}$-graded rational $G$-module, it is straightforward to check that $\invpt{\dual{V}} \cong \dual{\left( V^G \right)}$}.
\end{rem}

\begin{rem} \label{RdualE}
If $k$ is a field and $R$ is a Noetherian $\field{N}$-graded ring with $R_0 =k$, then $\dual{R} \cong E_R(k)$ as $R$-modules \cite[Proposition 3.6.16]{B-H}.
\end{rem}

\begin{rem}
If $k$ is a field and $R = k[x_1, \ldots, x_n]$ is a polynomial ring with homogeneous maximal ideal $\mathfrak{m}$, then $\dual{R} \cong E_R(k) \cong \LC{n}{\mathfrak{m}}{R}$ as $R$-modules.  However, if $\LC{n}{\mathfrak{m}}{R}$ is viewed as $R_{x_1 \ldots x_n} / \sum_{i=1}^n R_{x_1 \ldots \widehat{x_i} \ldots x_n}$, its grading is shifted:  $\LC{n}{\mathfrak{m}}{R} \cong \dual{R}(-n)$ as $\field{N}$-graded modules, where $\left[ \dual{R}(-n) \right]_j = \left[ \dual{R} \right]_{j-n}$.
\end{rem}

\begin{rem}[Matlis Duality for graded modules]  \label{MatlisDuality}
Suppose that $R$ is an $\field{N}$-graded ring such that $R_0 = k$, and that $M$ is a $\field{Z}$-graded $R$-module.  If $M$ has DCC (respectively, ACC) as a graded module, then $\dual{M}$ has ACC (respectively, DCC).  If $M$ has either DCC or ACC, the natural map $M \to \dbldual{M} $ is an isomorphism of graded modules.  Moreover, the functor $\dual{(-)}$ provides an anti-equivalence of categories from the category of $\field{Z}$-graded $R$-modules with DCC to the category of $\field{Z}$-graded $R$-modules with ACC, and vice versa \cite[Theorem 3.6.17]{B-H}.
\end{rem}

\begin{lem} \label{compatible}
Let $G$ be a linearly reductive group over a field $k$.  Let $R$ be an $\field{N}$-graded ring, with $R_0 = k$, that is also a $G$-module, where $G$ acts on $R$ by $k$-\edit{linear} automorphisms as to preserve its grading.  Let $M$ be a $\field{Z}$-graded $R[G]$-module such that the action of $G$ respects the grading on $M$.  If $M$ has DCC or ACC, then the action of $G$ on $M$ and the induced action on $\dbldual{M}$ are compatible under the natural isomorphism $M \overset{\cong}{\to} \dbldual{M}$ given by Matlis duality \textup{(}see \textup{Remark \ref{MatlisDuality}}\textup{)}; i.e.,  the isomorphism is an isomorphism of $G$-modules.. 
\end{lem}

\begin{proof}
Via $M \overset{\cong}{\to} \dbldual{M} = \bigoplus \limits_i \Hom_k(\Hom_k(M_i, k), k),$ $u = \sum \limits_i u_i \in M$ maps to $\sum \limits_i \phi_i$, where $\phi_i(f) = f(u_i)$ for $f \in  \Hom_k(M_i, k)$.  For $g \in G$, $((g \phi_i)(f))(u_i) = (\phi_i (g^{-1} f)) (u_i) = f (g u_i).$
\end{proof}

\noindent A straightforward calculation yields:

\begin{lem} \label{isom}
Suppose that $U$ and $V$ are $\field{Z}$-graded rational $G$-modules, where $\dim_k V$ and each $\dim_k U_i$ are finite.  Then $\dual{(U \otimes_k V)} \cong \dual{U} \otimes_k \findual{V}$ as $G$-modules, $\left[ \dual{(U \otimes_k V)} \right]_n$ precisely corresponding to $\left[ \dual{U} \otimes_k \findual{V} \right]_{-n}$.
\end{lem}

%%%%%%%%%%%%%%%%%%%%%%%%%%%%%%

\section{Proof of the Main Theorem} \label{mainthm}

We prove the Main Theorem \ref{general} in this section, which will be used (along with other tools) to prove the Main Theorem on Minors \ref{LC} in Section \ref{mainthmminors}.  Throughout this section, we will use the following definition and also refer to the subsequent frequently-used hypothesis.   

\begin{defn}[Socle]
Let $(R, \m, k)$ be a local ring or let $R$ be an $\field{N}$-graded ring with $R_0 = k$, a field, and homogeneous maximal ideal $\mathfrak{m}$.  Let $M$ be an $R$-module.   The \emph{socle of $M$}, denoted $\Soc M$, is the $R$-submodule $\Ann_M \mathfrak{m}.$ 
\end{defn}

\noindent Note that $\Soc M$ is naturally a  $k$-vector space.

\begin{hypothesis} \label{one}
Let $k$ be a field of characteristic zero, and let $R$ be a $\field{N}$-graded Noetherian ring such that $R_0 = k$, with homogeneous maximal ideal $\mathfrak{m}$.  Let $G$ be a linearly reductive group over $k$ acting on $R$ by $k$-\edit{linear} automorphisms such that $R$ is a $\ratGmod$, and let $M$ be a $\field{Z}$-graded $\ratRGmod$.  \edit{Moreover}, suppose that the actions of $G$ on $R$ and on $M$ respect their gradings.
\end{hypothesis}

\begin{rem}
Under Hypothesis \ref{one}, $\Soc M$ is a $\rat$ $R[G]$-submodule of $M$:  First we will see that it is a $G$-submodule.  For $g \in G$, $r \in \m$, and $u \in \Soc M$, since $G$ preserves the grading of $R$, $s:=g^{-1}r \in \m$, and since $M$ is an $R[G]$-module, $r  (g  u) = (g  s)(g  u) = g  (s  u) = g \cdot 0 = 0$. Since $\Soc M$ is also an $R$-submodule of $M$, it is also a $\ratRGmod$. 
\end{rem}

We will next state and prove Key Lemma \ref{hardlemma}, which implies a ``$\ratRGmod$ version" of the following theorem of Lyubeznik.

\begin{thm}[{\cite[Theorem 3.4]{Lyu}}] \label{Lyubeznik}Given a polynomial ring $R$ over a field $k$ of characteristic zero and ideals $I_1, \ldots, I_n$ of $R$, an iterated local cohomology module \[ M = \LC{i_1}{I_1}{ \LC{i_2}{I_2}{ \cdots ( \LC{i_n}{I_n}{R} )  \cdots }  }\] has only finitely many associated primes contained in a given maximal ideal of $R$.  If $M$ is supported only at a maximal ideal $\m$, then $M$ is isomorphic to a finite direct sum of copies of $E_R(R/\m)$. In particular, this holds when $M$ is a local cohomology module $\LC{i}{I}{R}$ that is supported only at $\m$, or when $M$ is any $\LC{0}{\m}{\LC{i}{I}{R}}$.
\end{thm}

\begin{keylemma} \label{hardlemma}
\edit{Suppose that $R, \mathfrak{m}, G$ and $M$ satisfy \textup{Hypothesis \ref{one}}, and that $M$ is also an injective Artinian $R$-module supported only at $\mathfrak{m}$.  Let $V = \Soc M.$  Then there exists a $G$-submodule $\widetilde{\findual{V} }$ of $\dual{M}$ \textup{(}see \textup{Definition \ref{gradeddual}}\textup{)} such that  $\widetilde{\findual{V}} \cong \findual{V}$ as $\ratGmod$s, and  \[ M \cong  \dual{ \left(R \tensor_k \widetilde{\findual{V}}  \right)}\] as $\ratRGmod$s, where $\dual{ \left( R \tensor_k \widetilde{\findual{V}}  \right)} \cong \dual{R} \tensor_k V$ as $\ratGmod$s.}
\end{keylemma}

\begin{rem}
\edit{In the statement of Key Lemma \ref{hardlemma}, $g \in G$ acts on $r \tensor u \in R \tensor_k \widetilde{\findual{V}}$ by $g  (r \tensor u) = gr \tensor gu$, and if $s \in R$, $s  (r \tensor u) = sr \tensor u.$ Note that $\widetilde{\findual{V}}$ is \emph{not} (necessarily) an $R$-module. }
\end{rem}

\begin{proof}
If $x_1, \ldots, x_n$ generate $\mathfrak{m}$, we have the exact sequence of $R$-modules: \[ \xymatrix{ &0 \ar[r] &V \ar[r]^i &M \ar[r]^{\theta} &M^{\oplus n},}\] where $i$ is the inclusion of $\ratRGmod$s, and $\theta(u) = (x_1 u, \ldots , x_n u)$ for $u \in M$.  By taking graded duals, we obtain the following exact sequence of $R$-modules: \[\xymatrix{ &\dual{\left(M^{\oplus n}\right)}  \ar[r]^{\dual{\theta}} & \dual{M} \ar[r]^{\dual{i}} & \dual{V} \ar[r] & 0,} \] where $\dual{i}$ is also a map of $\ratRGmod$s.  

\edit{For some $\alpha \in \field{N}$, $M \cong E_R(k)^{\oplus \alpha}$ as $R$-modules. Since $\Ann_{E_R(k)}{\m} = k$, $V= \Soc M$ must be a finite-dimensional $k$-vector space of dimension $\alpha$, so $\dual{V} = \findual{V}$.}

Under the canonical isomorphism $\dual{\left(M^{\oplus n}\right)} \cong  \left(\dual{M}\right)^{\oplus n}$, for $f_1, \ldots, f_n \in \dual{M}$,  $\dual{\theta} (f_1, \ldots, f_n) = x_1 f_1 + \ldots + x_n f_n.$  This means that $\image (\dual{\theta}) = \mathfrak{m} \dual{M},$ and \[\xymatrix{ & 0 \ar[r] & \mathfrak{m} \dual{M}  \ar[r] & \dual{M} \ar[r]^{\dual{i}} & \findual{V} \ar[r] & 0}\] is an exact sequence of $\ratRGmod$s.
%, so $\findual{V} \cong \dual{M}/\mathfrak{m} \dual{M}$ as $\rat$ $\RGmod$s.  
Moreover, since $G$ is linearly reductive, the map $\dual{i}$ has a $G$-module map splitting, $\phi$.  If $\widetilde{\findual{V}} \subseteq \dual{M}$ \edit{denotes} the image of $\phi$ (\edit{a $G$-module, but \emph{not} necessarily an $R$-module}), then \edit{$\widetilde{\findual{V}} = \dual{M}/\mathfrak{m} \dual{M}$ as $\ratGmod$s.}  

By the universal property of base change, the $k$-linear inclusion $\widetilde{\findual{V}} \inj \dual{M}$ induces a map of $R$-modules $\psi: R \tensor_k  \widetilde{\findual{V}} \to \dual{M}$ such that $\sum_i r_i \tensor v_i \mapsto \sum_i r_i v_i$.  Since $M$ is Artinian, $\dual{M}$ is Noetherian, and by Nakayama's lemma, a $k$-basis for $\widetilde{\findual{V}}\subseteq \dual{M}$ generates $\dual{M}$ minimally as an $R$-module, so $\psi$ is surjective.  For $\sum \limits_i r_i \tensor v_i \in R \tensor_k  \widetilde{\findual{V} }$, $g \cdot  \sum \limits_i r_i \tensor v_i =  \sum \limits_i g \cdot r_i \tensor g \cdot v_i$, and since $\dual{M}$ is an $R[G]$-module, $g \cdot \sum \limits_i r_i v_i  = \sum \limits_i (g \cdot r_i) (g \cdot v_i)$ in $\dual{M}$, so $\psi$ is $G$-equivariant.

\edit{Since $\dim_k \widetilde{\findual{V}} = \dim_k V = \alpha$, and $\dual{M} \cong R^{\oplus \alpha}$} (see Remark \ref{RdualE}), $\psi$ must be an isomorphism.  Noting Lemma \ref{compatible}, by taking graded duals, we have that $M \cong \dbldual{M} \cong \dual{\left(R \tensor_k  \widetilde{\findual{V} }\right)}$ as $\ratRGmod$s.   Moreover, by Lemma \ref{isom}, as $\ratGmod$s, $\dual{\left(R \tensor_k  \widetilde{\findual{V} }\right)} \cong \dual{R} \otimes_k \findual{ \left(\widetilde{\findual{V} } \right) } \cong  \dual{R} \otimes_k \dblfindual{V} \cong  \dual{R} \otimes_k V.$
\end{proof}

\begin{lem} \label{invzero}
Let $G$ be a linearly reductive group over a field $k$ and let $R$ be a $k$-vector space that is a $\field{Z}$-graded $G$-module, such that $G$ preserves its grading and $\dim_k R_i < \infty$ for all $i \in \field{Z}$.  Suppose that $V$ is a $G$-module.  If some simple $G$-submodule of $V$ is a $G$-submodule of $R$, then \[ (\dual{R} \otimes_k V)^G=0  \iff \dual{R} \otimes_k V=0.\]  In particular, if $R$, $\mathfrak{m}$, $G,$ and $M$ satisfy \emph{Hypothesis \ref{one} }, and $M$ is also an injective Artinian $R$-module supported only at $\mathfrak{m}$, then $ M^G = 0 \iff M = 0.$
\end{lem}

\begin{proof}
The backward implication clearly holds.  For the forward implication, suppose that $\dual{R} \otimes_k V \neq 0$ and that a simple $G$-submodule $W$ of $V$ is also a $G$-submodule of $R$, so that $W \inj R_n$ as $G$-modules for some $n$.  Dualizing, $\dual{R} \supseteq \findual{R_n} \surj \findual{W},$ which splits as $G$-modules since $G$ is linearly reductive, so $\findual{W} \inj \dual{R}$.  Thus, by Lemma \ref{tensorinvariants}, $(\dual{R} \otimes_k V)^G \neq 0$.

The last statement can be seen by applying the result to the case when $V = \Soc M$ and noting Lemma \ref{hardlemma}.
\end{proof}

\begin{lem} \label{indecomp}
Suppose that $R$, $\mathfrak{m}$, $G$, and $M$ satisfy \textup{Hypothesis \ref{one}} and that $M$ is also a nonzero injective Artinian $R$-module supported only at $\mathfrak{m}$.  Assume that all simple $G$-submodules of $\Soc M$ are also $G$-submodules of $R$.  If $\Soc M = V_1 \oplus \ldots \oplus V_\alpha$ as $G$-modules, where each $V_i$ is nonzero, then $M =  (\dual{R} \tensor_k V_1) \oplus \ldots \oplus (\dual{R} \tensor_k V_\alpha)$ as $\ratRGmod$s, where each $\dual{R} \tensor_k V_i$ is nonzero.  Moreover, $M^G = (\dual{R} \tensor_k V_1)^G \oplus \ldots \oplus (\dual{R} \tensor_k V_\alpha)^G$ as $R^G$-modules, where each $(\dual{R} \tensor_k V_i)^G$ is nonzero.  In particular, if $M^G$ is a indecomposable $R^G$-module, then $\Soc M$ is a simple $G$-module.  
\end{lem}

\begin{proof}
The first statement follows from Lemma \ref{hardlemma} after applying $\dual{R} \tensor_k (-)$.  The second follows by applying $(-)^G$ and noting that each summand is nonzero by Lemma \ref{invzero}.
\end{proof}

\begin{lem}  \label{GactionLC}
Let $G$ be a linearly reductive group over a field $k$ acting on a Noetherian $k$-algebra $R$, and let $N$ be a $\ratRGmod$.  Let $J \subseteq R^G$ be an ideal, and let $I = J R$.  Then every $\LC{i}{I}{N}$ is also a $\ratRGmod$, and every simple $G$-submodule of $\LC{i}{I}{N}$ is also a $G$-submodule of $N$.  Moreover, there is a canonical isomorphism of $R^G$-modules \[\invpt{\LC{i}{I}{N}} \cong \LC{i}{J}{N^G}.\]
\end{lem}

\begin{proof}
\edit{Since $N$ is a $G$-module, for $f \in R^G$, $N_f = \lim \limits_{\longrightarrow} \left(N \overset{\cdot f}{\to} N \overset{\cdot f}{\to} N \overset{\cdot f}{\to} \ldots \right)$ as $G$-modules:  if $g \in G$ and $\frac{u}{f^m} \in N_f,$ which corresponds to $\left[ u \right]$ in the $m$th copy of $N$ in the direct limit, then $g \cdot \frac{u}{f^m} = \frac{g \cdot u}{f^m}$, which corresponds to $[g \cdot u]$ in the $m$th copy of $N$ in the direct limit.  This makes $N_f$ a rational $G$-module, as $N$ is one, and all simple $G$-submodules of $N_f$ are also $G$-submodules of $N$.  }

Say that $J = (f_1, \ldots, f_n)$. Since products of any of the $f_j$ are fixed by $G$, every term in the following complex is a rational $G$-module: \[ \xymatrix{ 0 \ar[r] & N  \ar[r]^(.4){\delta_0} & \bigoplus \limits_{j=1}^n N_{f_j} \ar[r]^(.55){\delta_1} & \ldots \ar[r]^(.3){\delta_{n-2}} & \bigoplus \limits_{j=1}^n N_{f_1 \ldots \widehat{f_j} \ldots f_n} \ar[r]^(.55){\delta_{n-1}} &N_{f_1 f_2 \ldots f_n} \ar[r] & 0. } \]

\noindent Since on each summand, the maps $\delta_j$ are, up to a sign, further localization maps, they are $G$-equivariant.  This makes the cohomology modules $\LC{i}{I}{N}$ rational $G$-modules as well, and they inherit the property that all their simple $G$-submodules are also $G$-submodules of $N$.  Additionally, these local cohomology modules are $R[G]$-modules since $N$ is one:  given any $g \in G$, $r \in R$, and $\left[ \frac{u}{f^m} \right] \in \LC{i}{I}{N}$, $g \left( r \left[ \frac{u}{f^m} \right]  \right) = g \left[ \frac{r u}{f^m} \right] = \left[ \frac{(gr)(gu)}{f^m} \right] = (g r)\left(g \left[ \frac{u}{f^m} \right] \right).$  Thus, they are $\ratRGmod$s.

For the last statement, first notice that for any $f \in R^G$, $ \left(N^G\right)_f =  \left( N_f\right)^G$.  As taking invariant parts commutes with direct sums, $\LC{i}{J}{N^G}$ is isomorphic the cohomology of the complex \[ \xymatrix @C.2in{ 0 \ar[r] & N^G  \ar[r]^(.35){d_0} &  \invpt{ \bigoplus \limits_{j=1}^n N_{f_j} } \ar[r]^(.65){d_1} & \ldots \ar[r]^(.25){d_{n-2}} & \invpt{ \bigoplus \limits_{j=1}^n N_{f_1 \ldots \widehat{f_j} \ldots f_n} } \ar[r]^(.57){d_{n-1}} & \invpt{ N_{f_1 f_2 \ldots f_n}}\ar[r] & 0. } \] where $d_i$ is the restriction of $\delta_i$ to the invariant part of the $i^\text{th}$ module in the complex.  Since $G$ is linearly reductive, the functor $V \mapsto V^G$ of $G$-modules is exact (see Remark \ref{invexact}),  and we conclude that  $(\LC{i}{I}{N})^G \cong \LC{i}{J }{N^G}.$
\end{proof}

\begin{main} \label{general} \label{main} Let $R$ be a polynomial ring over a field $k$ of characteristic zero with homogeneous maximal ideal $\m$.  Let $G$ be a linearly reductive group over $k$ acting by degree-preserving $k$-automorphisms on $R$, such that $R$ is a $\ratGmod$.  Assume that $A = R^G$ has homogeneous maximal ideal $\m_A$, let $d = \dim A$, and let $I = \m_A R$.  Then $\LC{d}{I}{R} \neq 0$ and $I$ is generated up to radicals by $d$ elements and not fewer, so that $\LC{i}{I}{R} = 0 \text{ for } i>d.$  Moreover, the following hold:

\begin{enumerate}
\item[\textup{(1)}] If $i < d$, then $\m$ is not an associated prime of $\LC{i}{I}{R}$; i.e., $\LC{0}{\m}{\LC{i}{I}{R}}=0$.
\end{enumerate}  

\noindent If $\LC{d}{I}{R}$ is supported only at $\m$ \emph{(}e.g., this holds if, after localization at any of the indeterminates of $R$, $I$ requires fewer than $d$ generators up to radical\emph{)}, then
\begin{enumerate}
\item[\textup{(2)}] If $H := \LC{d}{I}{R}$, then $V := \Soc H$ is a simple $G$-module, and 
\item[\textup{(3)}] There exists a $G$-submodule $W$ of $H$ such that $H \cong \dual{ \left( R \tensor_k W \right) }$ as rational $R[G]$-modules, which, as $\ratGmod$s, is isomorphic to $E_R(k) \tensor_k V$ \emph{(}where the action of $G$ on $E_R(k) \cong \dual{R}$ is induced by its action on $R$\emph{)}. 
\end{enumerate}  
\end{main}

\begin{proof}
By Lemma \ref{GactionLC}, we know that for every $i$, $(\LC{i}{I}{R})^G \cong \LC{i}{m_A}{A}$ as $R^G$-modules.   The invariant part of $\LC{d}{I}{R}$, $\LC{d}{m_A}{A}$, is nonzero since $d = \dim A$, so $\LC{d}{I}{R} \neq 0.$  The maximal ideal $\m_A$ of $A$ is generated, up to radical, by $d = \dim A$ elements, so its expansion to $R$, $I = \m_A R$, will also be generated up to radical by the same $d$ elements. 

For (1), assume that $i < d$. By Lemma \ref{GactionLC}, $\LC{i}{I}{R}$ is a $\ratRGmod$, so that its submodule of elements killed by some power of $\mathfrak{m}$, $\LC{0}{\mathfrak{m}}{ \LC{i}{I}{R} }$, is also a $\ratRGmod$.  By definition, $\left( \LC{0}{\mathfrak{m}}{ \LC{i}{I}{R} }\right)^G$ is the $R^G$-submodule of $\LC{i}{I}{R}$ consisting of invariant elements that are killed by some power of $\m$; thus, it is the $R^G$-submodule of $\LC{i}{I}{R}^G \cong \LC{i}{\m_A}{A}$ (by Lemma \ref{GactionLC}) consisting of elements killed by a power of $\m$.  By the theorem of Hochster and J. Roberts \cite[Main Theorem]{H-R} or of Boutot \cite[Th\'eor\`eme]{Boutot}, since $G$ is linearly reductive and $A=R^G$, $A$ must be Cohen-Macaulay.  Since $i < d = \dim A$, $\LC{i}{ \mathfrak{m}_A}{A} = 0$, so in particular, its submodule $\left( \LC{0}{\mathfrak{m}}{ \LC{i}{I}{R} }\right)^G$ must also vanish.  By Theorem \ref{Lyubeznik}, $\LC{0}{\mathfrak{m}}{ \LC{i}{I}{R} }$ is isomorphic to a finite direct sum of copies of $E_R(k).$  Therefore, by Lemma \ref{invzero}, since its invariant part vanishes, $\LC{0}{\mathfrak{m}}{ \LC{i}{I}{R} }$ must also vanish.

Now suppose that $\LC{d}{I}{R}$ is supported only at $\mathfrak{m}$.  Since $R$ is a domain (so $A$ also is), $\omega_A$, the canonical module of $A$, must be rank one and isomorphic to an ideal of $A$, so is torsion free \cite[Proposition 3.3.18]{B-H}.  Thus, $\omega_A$ is an indecomposable $A$-module.  Therefore, \[  \Hom_A\left(\omega_A, E_A(A/\m_A)\right) \cong \LC{d}{\mathfrak{m}_A}{A} \cong (\LC{d}{I}{R})^G  \]  must also be indecomposable, so $\Soc \LC{d}{I}{R}$ is a simple $G$-module by Lemma \ref{indecomp}.

Part (3) is an application of Key Lemma \ref{hardlemma}.
\end{proof}

%%%%%%%%%%%%%%%%%%%%%%%%%%%%%%

\section{Proof of the Main Theorem on Minors} \label{mainthmminors}

\begin{rem} \label{groupactionclosed}
If a topological group $G$ acts continuously on a topological space $Z$ permuting a finite collection of closed sets $V_1, \ldots, V_m \subseteq Z$, then $G$ must fix each $V_i$:  
For each $1 \leq i \leq m$ and $v \in V_i$, the map $\Theta_{v} : G \to Z $ given by $\Theta_{v}(g) = g \cdot v$ is continuous, so the set $\Theta_{v}^{-1} (V_i) = \{ g \in G \ | \ g \cdot v \in V_i \}$  is closed in $G$.   Similarly, if $\Theta^{v} : G \to \field{A}^n$ is given by $\Theta^{v}(g) = g^{-1} \cdot v$, then the set $(\Theta^{v})^{-1}(V_i)$ is also closed.  Thus, the sets $\bigcap \limits_{v \in V_1} \Theta_{v}^{-1} (V_i) = \{ g \in G \ | \ g V_i \subseteq V_i \}$ and $\bigcap \limits_{v \in V_1} (\Theta^{v})^{-1} (V_i) = \{ g \in G \ | \ V_i \subseteq g \cdot V_i \}$  are closed in $G$, and so their intersection, $\{ g \in V_i \ | \ g \cdot V_i = V_i \} = \stab_G V_i$ (the stabilizer of $V_i$ in $G$) is also closed in $G$.  

As $G$ permutes the $V_i$, we have a map $\phi : G \to S_m$, $S_m$ the symmetric group on $m$ letters. Since $\phi^{-1}(\stab_G (V_i)) = \stab_{S_m} (i)$, $\phi$ induces $G/\stab_G (V_i) \inj S_m/ \stab_{S_m} (i)$. As  $S_m/ \stab_{S_m} (i)$ is finite, so is $G/\stab_G (V_ i)$, and each $\stab_G (V_i)$ is a finite index subgroup of $G$.  If $\stab_G (V_i) \subsetneq G$, then since $G$ is closed, its cosets would disconnect $G$, which is impossible.  Thus, each $\stab_G (V_i) = G$, and $G$ fixes each $V_i$.
\end{rem}

\begin{lem} \label{stableprimes}
Let $G$ be a connected linear algebraic group, let $R$ be a rational $G$-module, and let $M$ be an $R[G]$-module such that $\Ass{R}{M}$ is finite.  Then every associated prime of $M$ is stable under the action of $G$. 
\end{lem}

\begin{proof}
It is easy to check that if $u \in M$ and $\p = Ann_R u$,  then $g\cdot \p = Ann_R (g \cdot u)$.  Hence,  $G$ permutes the finite set of associated primes $\p_i$, and, consequently, the closed sets $\field{V}(\p_i)$.  The result now follows from Remark \ref{groupactionclosed}.
\end{proof}

\begin{hypothesis} \label{hypminors}
Let $k$ be a field of characteristic zero, let $X$ be an $r \times s$ matrix of indeterminates, where $r < s$, and let $R = k[X]$ be the polynomial ring over $k$ in the entries of $X$.  For $0 < t \leq r$, let $\I{t}{X}$ be the ideal of $R$ generated by the $t \times t$ minors of $X$, which is prime by \emph{\cite[Theorem 1]{Hoch-Eag}}.  Furthermore, let $I = \I{r}{X}$ be the ideal generated by the maximal minors of $X$.
\end{hypothesis}

\begin{rem}[Square matrix case]
Suppose that $R$ satisfies Hypothesis \ref{hypminors}, but assume instead that $r = s$.  Here, $I = \left( \Delta \right)$, where $\Delta$ is the determinant of $X$, and the only nonzero local cohomology module is $\LC{1}{I}{R}$, which is isomorphic to $R_\Delta/R.$  
\end{rem}

\begin{rem}[Action of the special linear group on our polynomial ring] \label{SLaction}
Let $k$, $R$, and $I$ satisfy Hypothesis \ref{hypminors}, and let $G = SL_r(k)$, which is linearly reductive since the characteristic of $k$ is zero.  Considering $\Gamma \in G$ as an $r \times r$ matrix, the action of $\Gamma$ on the $k$-algebra $R$ is defined by where the entries of $X$ are sent.  The  $(\alpha, \beta)^{\text{th}}$ entry of $X$ is sent to the $(\alpha, \beta)^{\text{th}}$ entry of $\Gamma^{-1} \cdot X$.  Thus, $G$ acts by $k$-algebra automorphisms that correspond to invertible row operations on the matrix $X$.  Additionally:

\begin{enumerate}
\item The maximal minors of $X$ are fixed by the action of $G$, so the ideal $I$ generated by them is $G$-stable.  A classical invariant theory result of Weyl states that, in fact, $R^G$ is the $k$-subalgebra of $R$ generated over $k$ by the maximal minors of $X$  \cite[Theorem 2.6.A]{Weyl}.  Thus,  if $\m_{R^G}$ is the homogeneous maximal ideal of $R^G$, $I = \m_{R^G} R$. 
\item In fact, $R^G$ is the homogeneous coordinate ring of the Pl\"ucker embedding of the Grassmann variety of $r$-planes in $s$-space, which has dimension $r(s-r)$; therefore, $ \dim R^G = r(s-r)+1.$
\end{enumerate} 
\end{rem}

\begin{rem}
Under Hypothesis \ref{hypminors}, $R$ is a $\ratRGmod$:  Since the action of $G$ is induced by that on the linear forms, $R$ is certainly an $R[G]$-module.  Moreover, $R$ will be the directed union of $V_n := \bigoplus \limits_{i\leq n} R_i$, each a finite-dimensional $G$-module.
\end{rem}

\begin{claim} \label{simple}
Suppose that $k$, $R$, and $I$ satisfy \textup{Hypothesis \ref{hypminors}}, and let $d = r(s-r)+1$.  Then $\dim_k \Soc \LC{d}{I}{R} = 1$.
\end{claim}

\begin{proof}
Let $G = SL_r(k)$ act on $R$ as in Remark \ref{SLaction}.  Let $\m_{R^G}$ be the homogeneous maximal ideal, and let $V = \Soc \LC{d}{I}{R}$.  Recall that $\dim R^G = d$. Thus, as $G$ is a connected semisimple group over a field of characteristic zero, and preserves degrees in its action on $R$, $R^G$ is Gorenstein by the theorem of Hochster and J. Roberts \cite[Corollary 1.9]{H-R}.  Therefore, its canonical module is isomorphic to $R^G$, so by Matlis duality (see Remark \ref{MatlisDuality}) $\dual{\left(\LC{d}{\m_{R^G}}{R^G}\right)} \cong R^G$, with a possible shift in grading.  Moreover, by taking $M = \LC{d}{I}{R}$ in Key Lemma \ref{hardlemma} and applying graded duals, $\dual{\left( \LC{d}{I}{R} \right)} \cong R \tensor_k \findual{V}$ as $\ratGmod$s.  As $G$-modules,  

\begin{align}  
R^G &\cong \dual{ \left(  \LC{d}{\m_{R^G}}{R^G} \right) } \notag \\
& \cong \dual{\left( \invpt{\LC{d}{I}{R}} \right) } \label{lemmaLC} \\
& \cong \invpt{ \dual{\left( \LC{d}{I}{R} \right) } } \label{linred} \\
%&\cong \invpt{ \dual{\left(\dual{R} \tensor_k V \right)} } \\
%&\cong \invpt{ \dbldual{\left( R \tensor_k \findual{V}  \right)}  } \\
&\cong \invpt{ R \tensor_k \findual{V}}, \notag
\end{align} 
where \eqref{lemmaLC} holds by Lemma \ref{GactionLC} and \eqref{linred} since $G$ is linearly reductive (see Remark \ref{dualinvcommute}).

By Main Theorem \ref{main} (2), $V = \Soc \LC{d}{I}{R}$ is a simple $G$-module. Let $W$ denote the $V$-isotypical component of $R$ (see Definition \ref{isotyp}).  Note that since $V$ is a simple $G$-module and $G$ respects the grading on $R$, each submodule of $R$ isomorphic to $V$ sits in one degree.  Then, $W$ a $G$-submodule of $R$, has a natural induced grading; each graded piece, $W_j \subseteq R_j$, is isomorphic as a $G$-module to a finite direct sum of copies of $V$, say $W \cong V^{\oplus{n_j}}$.  We have the following $G$-module isomorphisms (where \eqref{isotypic} is due to Lemma \ref{tensorinvariants}.
%and \eqref{last} is as in Remark \ref{vsisom}): 
\begin{align} R^G &\cong \invpt{ \findual{V} \tensor_k R } \notag \\ \label{isotypic} &\cong \invpt{ \findual{V} \tensor_k \left( \bigoplus_j  V^{\oplus n_j} \right) } \\ \notag &\cong \bigoplus_j \left( \invpt{ \findual{V} \tensor_k  V } \right)^{\oplus n_j} \end{align}
%\\ & \label{last} \cong  \bigoplus_j \left( \invpt{ \Hom_k(V, V) } \right)^{\oplus n_j} \\ \notag &\cong \bigoplus_j \left( \Hom_G(V, V) \right)^{\oplus n_j}.

\noindent It is easily checked that, up to a shift, grading is preserved under these maps. If $\mu$ denotes the smallest degree for which $\left( \invpt{ \findual{V} \tensor_k  V } \right)^{\oplus n_\mu} \neq 0$, then $\mu$ is the smallest degree for which $V$ injects as a $G$-module into $R_\mu$.  Since $\dim_k \left[R^G\right]_0 = 1$, $\dim_k\left( \invpt{ \findual{V} \tensor_k  V } \right)^{\oplus n_\mu} = 1$, and we must have that $n_\mu = 1$. 
 
%we have that $\Hom_G(V, V) \cong k$.  

With our specific choice of $G$ and $R$, if any simple $G$-module with a nontrivial action is a $G$-submodule of $R$, it occurs with multiplicity greater than one in the smallest degree of $R$ in which it occurs.  
This is, for example, a consequence of \cite[Theorem 5.2.7]{Goodman}.  
Since $n_\mu =1$, we know that $G$ must act trivially on $V$.  
Since $V$ is a simple $G$-module, we have that $\dim_k V = 1$.
\end{proof}

\begin{rem}[Another useful group action on our polynomial ring] \label{SLSLaction}
Let $k$ and $R$ satisfy Hypothesis \ref{hypminors}, and let $H$ be the connected group $SL_r(k) \times SL_s(k)$  \cite[Theorem 2.19]{Goodman}.  Then $H$ acts on $R$ by $k$-algebra automorphisms as follows:  Considering $\Gamma \in SL_r(k)$ and $\Gamma^\prime \in SL_s(k)$ as $r \times r$ and $s \times s$ matrices, respectively, the action of $\Gamma \times \Gamma'$ sends the entries of $X$ to those of $\Gamma^{-1} X \Gamma^{\prime}$.  The action of $H$ is clearly transitive on the entries of the matrix.
\end{rem}

The following observation is used in the proof of the Main Theorem on Minors \ref{LC}.

\begin{rem} \label{localizeisom}
Let $k$, $R$, $I$, and $\I{t}{X}$ satisfy Hypothesis \ref{hypminors}.  Over $R_{x_{1 1}},$ we can perform elementary row and column operations on $X$ to obtain the matrix  \[ \begin{bmatrix}
1 &\vline& 0 & \ldots & 0\\ 
\hline 0 & \vline &  &\\ 
\vdots & \vline &  & Y\\
0 & \vline & &\\
\end{bmatrix}, \]  where $Y$ is the $( r-1 ) \times ( s -1 )$ matrix $\left[x_{\alpha \beta} - \frac{x_{1 \beta}}{x_{1 1}}{x_{\alpha 1}} \right]_{\substack{1 < \alpha \leq r \\ 1 < \beta \leq s}}.$

Let $S = k\left[ y_{\alpha \beta}, x_{1 1}, x_{1 1}^{-1}, x_{\alpha 1}, x_{1 \beta} \ | \ 1 < \alpha \leq r, 1 < \beta \leq s \right].$  Since the operations used to transform the matrix are invertible, the transformation defines an isomorphism \begin{align}
  S &\cong R_{x_{1 1}}, \text{ where} \label{mxisom} \\
  y_{\alpha \beta} & \mapsto x_{\alpha \beta} - \frac{x_{1 \beta}}{x_{\alpha 1}}{x_{1 1}}. \notag
\end{align}  Under \eqref{mxisom}, the ideal $\I{t+1}{X}$ of $R_{x_{1 1}}$ corresponds precisely to the ideal $\I{t}{Y}$ of $S$.  Thus, this map induces an isomorphism $\left(\LC{i}{\I{r}{X}}{R}\right)_{x_{1 1}} \cong \LC{i}{\I{r-1}{Y}}{S}$ for any $i$.  Since $ \LC{i}{\I{r-1}{Y}}{S} \cong \LC{i}{\I{r-1}{Y}}{k[Y] \tensor_{k[Y]} S}$, which is, in turn, isomorphic to $ \LC{i}{\I{r-1}{Y}}{k[Y]} \tensor_{k[Y]} S$ since $S$ is flat over $k[Y]$, we have that \begin{align} \label{RSLC} \left(\LC{i}{\I{r}{X}}{R}\right)_{x_{1 1}} &\cong \LC{i}{\I{r-1}{Y}}{k[Y]} \tensor_{k[Y]} S.\end{align}
\end{rem}

\begin{mainminors} \label{LC} 
Suppose that $k$, $R$, $I$, and $\I{t}{X}$ satisfy \textup{Hypothesis \ref{hypminors}}.

\begin{enumerate}
\item[\textup{(1)}]  For $d=r(s-r)+1$,  $\LC{d}{I}{R} \cong E_R(k).$ 
\item[\textup{(2)}] $\LC{i}{I}{R} \neq 0$ if and only if $i=(r-t)(s-r) + 1$ for some $0 \leq t < r.$
\item[\textup{(3)}] Furthermore, if $i= (r-t)(s-r)+1$, then \[\LC{i}{I}{R} \hookrightarrow E_R(R/ \I{t+1}{X}) \cong \LC{i}{I}{R}_{\I{t+1}{X}}.\] \noindent In particular, $\Ass{R}{\LC{i}{I}{R}} = \{ \I{t+1}{X} \}.$\end{enumerate}
\end{mainminors}

\begin{proof}
First consider $d = r(s-r)+1$, the dimension of the invariant ring $R^G$ under the action of $G$ from Remark \ref{SLaction}.  By Main Theorem \ref{main}, $\LC{i}{I}{R} = 0$ for any $i>d$.  Applying this again to the smaller matrix $Y$ from Remark \ref{localizeisom}, we see that $\LC{i}{\I{r-1}{Y}}{k[Y]} = 0$ if $i>(r-1)\left( (s-1) - (r-1)\right) + 1$; in particular,  $\LC{d}{\I{r-1}{Y}}{k[Y]} = 0$.  Therefore, \eqref{RSLC} indicates that $\left(\LC{d}{\I{r}{X}}{R}\right)_{x_{1 1}} = 0.$  By symmetry, $\LC{d}{\I{r}{X}}{R}$ vanishes after localizing at any $x_{\alpha \beta}$, and so $\LC{d}{\I{r}{X}}{R}$ is supported only at the homogeneous maximal ideal $\m$ of $R$.

Thus, by Theorem \ref{Lyubeznik}, $\LC{d}{I}{R} \cong E_R(k)^{\oplus \alpha}$ for some  integer $\alpha$.  As $\Ann_{E_R(k)}{\m} = k$, $\Soc \LC{d}{I}{R} = \Ann_{\LC{d}{I}{R}} \m$ is a $k$-vector space of dimension $\alpha$.  By Proposition \ref{simple}, $\alpha = 1$, proving (1).

We now use induction on $r$, for all $s \geq r$, to prove that if $ i=(r-t)(s-r) + 1$ for some $0 \leq t < r$, $X$ is an $r \times s$ matrix of indeterminates, $k$ is a field of characteristic zero, and $R = k[X]$, then  \begin{align}
\Ass{R}{\LC{i}{I}{R}} = \{ \I{t+1}{X} \}, \text{ and} \label{assprove}  \\
\LC{i}{I}{R}_{\I{t+1}{X}} \cong E_R(R/ \I{t+1}{X}) \label{locprove},
 \end{align} and for $i$ not of this form, $\LC{i}{I}{R}$ vanishes.  (This would prove (2) and (3).)

For the base case, let $r =1$.  Here, $R = k[ x_1, \ldots, x_s]$; if $t=0$, then $i = r(s-r)+1 = s$.  Since $I = \I{1}{X}$ is the homogeneous maximal ideal of $R$, $\LC{s}{ I }{R} \cong E_R(k)$, and $\LC{i}{I}{R}=0$ for all $i \neq s$.

Now say that for all $r_0 < r$ and $s_0 \geq r_0,$ for any $0 \leq t_0 < r_0$, if $R = k[X]$, where $k$ is a field of characteristic zero and $X = [x_{\alpha \beta}]$ is an $r_0 \times s_0$ matrix of indeterminates, and $i = (r_0-t_0)(s_0-r_0)+1,$ then we have that $\Ass{R}{ \LC{i}{\I{r_0}{X}}{R}} = \{ \I{t_0+1}{X} \}$ and $\left(\LC{i}{\I{r_0}{X}}{R}\right)_{\I{t_0+1}{X}} \cong E_R\left(R/\I{t_0+1}{X}\right)$. Assume, moreover, that for all $i$ not of this form, $\LC{i}{\I{r_0}{X}}{R}$ vanishes.

Take $X$ an $r \times s$ matrix of indeterminates, $R = k[X]$, and $I = \I{r}{X}$.  In proving (1), we have already shown \eqref{assprove} and \eqref{locprove} for $i = d = r(s-r)+1$.    For $i< d$, $\m$ is not an associated prime of $\LC{i}{I}{R}$ by Main Theorem \ref{general} (1), so some $x_{\alpha \beta}$ must be a nonzerodivisor on $\LC{i}{I}{R}$.  We could renumber the indeterminates to assume that $x_{11}$ is nonzerodivisor, but, in fact, each $x_{\alpha \beta}$ is a nonzerodivisor on $\LC{i}{I}{R}$:  Consider the action of the group $H$ described in Remark \ref{SLSLaction}.  Since $H$ is connected and $\Ass{R}{ \LC{i}{I}{R} }$ is finite by Theorem \ref{Lyubeznik} (due to the grading on $R$, each associated prime is contained in $\m$), Lemma \ref{stableprimes} implies that each associated prime of $\LC{i}{I}{R}$ is stable under its action.  Since every indeterminate $x_{\alpha \beta}$ is in the orbit of every other indeterminate, because some $x_{\alpha \beta}$ is a nonzerodivisor, every one is a nonzerodivisor.

By the inductive hypothesis, all $\LC{i}{\I{r-1}{Y}}{k[Y]} = 0$ unless $0 \leq t_0 < r-1$ and \[ i = \left((r-1)-t_0)((s-1)-(r-1)\right)+1 = (r-1-t_0)(s-r)+1,\] or, equivalently, $i  = (r-t)(s-r)+1$ with $1 \leq t < r$.  Since each such $i$ is less than $d$, $x_{11}$ is a nonzerodivisor on $\LC{i}{I}{R}$, and \eqref{RSLC} implies that the same vanishing conditions must hold for the $\LC{i}{\I{r}{X}}{R}$.  Combining this fact with (1), we see that  $\LC{i}{\I{r}{X}}{R}$ must vanish for $i < d$ unless $i = (r-t)(s-r)+1$ for some $0 \leq t < r.$

Suppose that $i = (r-t)(s-r) + 1$ for some $t > 0.$  The inductive hypothesis tells us that $\Ass{k[Y]}{\LC{i}{\I{r-1}{Y}}{k[Y]}} = \{ \I{t}{Y} \},$ and since $S$ is flat over $k[Y]$ \cite[Theorem 12]{Mat}, $\Ass{S}{\LC{i}{\I{r-1}{Y}}{k[Y]} \tensor_{k[Y]} S} = \{ \I{t}{Y}S \}.$  Thus, \eqref{RSLC} implies that $\Ass{R}{\LC{i}{\I{r}{X}}{R}_{x_{11}}}$ consists solely of $\I{t+1}{X}$, the ideal that corresponds to $\I{t}{Y}S$ under \eqref{mxisom}.   Since $x_{11}$ is a nonzerodivisor on $\LC{i}{I}{R}$, the associated primes of $\left(\LC{i}{\I{r}{X}}{R}\right)_{x_{1 1}}$ are those of $\LC{i}{\I{r}{X}}{R}$ expanded to $R_{x_{11}}$, and $\Ass{R}{\LC{i}{\I{r}{X}}{R}} = \{ \I{t+1}{X} \},$ proving \eqref{assprove}.

Hochster and Eagon found that $\height_{R} I_t(X) = (r-t+1)(s-t+1)$, so $\height_{k[Y]} \I{t}{Y} = (r-t)(s-t) = \height_R \I{t+1}{X}$ \cite[Theorem 1]{Hoch-Eag}.  Therefore, noting that $x_{11} \notin \I{t+1}{X}$  for any $1 \leq t < r$, the following sequence of isomorphisms proves \eqref{locprove}:
 \begin{align} 
\label{isom1} \left(\LC{i}{\I{r}{X}}{R}\right)_{\I{t+1}{X}} &\cong \left(\LC{i}{\I{r-1}{Y}}{S} \right)_{\I{t}{Y}S} \\ 
\label{isom2}  &\cong \left(\LC{i}{\I{r-1}{Y}}{k[Y]}\right)_{\I{t}{Y}} \tensor_{k[Y]} S  \\ 
\label{isom3} &\cong E_{k[Y]}\left(k[Y]/\I{t}{Y}\right) \tensor_{k[Y]} S \\  
\label{isom5} &\cong \left( \LC{\height \I{t}{Y}}{\I{t}{Y} }{k[Y]} \right)_{\I{t}{Y}} \tensor_{k[Y]} S  \\ 
\label{isom6}  &\cong\LC{\height \I{t}{Y}}{\I{t}{Y} }{k[Y] \tensor_{k[Y]} S }_{\I{t}{Y}}  \\ 
\notag &\cong \LC{\height \I{t}{Y}}{\I{t}{Y}}{S}_{\I{t}{Y}}  \\ 
\notag &\cong \LC{\height \I{t+1}{X}}{\I{t}{Y} S_{\I{t}{Y}}}{S_{\I{t}{Y}}}  \\ 
\label{isom9} &\cong \LC{\height \I{t+1}{X}}{\I{t+1}{X}R_{\I{t+1}{X}} }{R_{\I{t+1}{X}}} \\
\label{isom10} &\cong E_R\left(R/\I{t+1}{X}\right). 
\end{align}  
\noindent  \eqref{isom1} and \eqref{isom9} are induced by \eqref{mxisom}, \eqref{isom2} and \eqref{isom6} occur because $S$ is flat over $k[Y]$, \eqref{isom3} is by the inductive hypothesis, and since $R$ is Gorenstein, we have \eqref{isom5} and \eqref{isom10}.  \end{proof}

%%%%%%%%%%%%%%%%%%%%%%%%%%%%%%

\section*{Acknowledgements}

This results of this article are part of the author's thesis completed at the University of Michigan.  She sincerely thanks her advisor, Mel Hochster, for sharing invaluable ideas, suggestions, and advice.  She would also like to thank Daniel Hern\'andez for useful mathematical discussions regarding this work, and for his support.

%%%%%%%%%%%%%%%%%%%%%%%%%%%%%%

\nocite{B-V, Bjork}

%-----------------------------
\bibliographystyle{alpha}
\bibliography{References}
%-----------------------------

\vspace{.2cm}

\begin{center}
\noindent \small \textsc{Department of Mathematics, University of Minnesota, Minneapolis, MN  55455} \\ \emph{Email address}:  \href{mailto:ewitt@umn.edu}{\tt ewitt@umn.edu} 
\end{center}

\end{document}